\documentclass[secthm,seceqn,amsthm,ussrhead,reqno,12pt]{amsart}
\usepackage[utf8]{inputenc}
\usepackage[english]{babel}
\usepackage[symbol]{footmisc}
\usepackage{amssymb,amsmath,amsthm,amsfonts,xcolor,enumerate,hyperref,comment,longtable,cleveref}

\usepackage{times}
\usepackage{cite}
\usepackage{pdflscape}
\usepackage{ulem}
\usepackage[mathcal]{euscript}
\usepackage{tikz}
\usepackage{hyperref}
\usepackage{cancel}
\usepackage{stmaryrd}

\usepackage{amsfonts}
\usepackage{amssymb}
\usepackage{times}
\usepackage{xcolor}
\usepackage{url}
\usepackage{float}

\usepackage{cite}
\usepackage{hyperref}

\usepackage{amsfonts}
\usepackage{amsmath}
\usepackage{eurosym}
\usepackage{geometry}

\usepackage{caption,booktabs}

\theoremstyle{plain}
\newtheorem{theorem}{Theorem}
\newtheorem{lemma}[theorem]{Lemma}

\newtheorem{definition}[theorem]{Definition}
\newtheorem{corollary}[theorem]{Corollary}
\begin{document}
 
 \medskip

\noindent{\Large
Degenerations of noncommutative Heisenberg algebras}\footnote{
The first part of this work is supported by  
FCT   UIDB/00212/2020 and UIDP/00212/2020. 
The second part of this work is supported by the Russian Science Foundation under grant 22-11-00081.  
The authors thank 
Mikhail Ignatyev and Yury Popov for some constructive discussions about the paper.
}

 \medskip

 \medskip

\begin{center}

 {\bf
Ivan Kaygorodov\footnote{CMA-UBI, Universidade da Beira Interior, Covilh\~{a}, Portugal; \  kaygorodov.ivan@gmail.com} \& 
 Yury Volkov\footnote{Saint Petersburg State University, Russia;\  wolf86\_666@list.ru}

}
 
\end{center}

\ 

\noindent {\bf Abstract.}
{\it We give the full description of all degenerations of complex five dimensional noncommutative Heisenberg algebras.
As a corollary, we have the full description of all degenerations of four dimensional anticommutative  $3$-ary algebras.}

\ 


\ 


 \medskip 

\section*{Introduction}
 
The geometry of varieties of algebras defined by polynomial identities has been an active area of interest and research since the works of Nijenhuis--Richardson~\cite{NR66} and Gabriel~\cite{gabriel} in the 1960's and 1970's. The relationship between geometric features of the variety (such as irreducibility, dimension, and smoothness) and the algebraic properties of its points brings novel geometric insight into the structure of the variety, its generic points and degenerations. Given algebras ${\mathcal A}$ and ${\mathcal B}$ in the same variety, we write ${\mathcal A}\to {\mathcal B}$ and say that ${\mathcal A}$ {\it degenerates} to ${\mathcal B}$, or that ${\mathcal A}$ is a {\it deformation} of ${\mathcal B}$, if ${\mathcal B}$ is in the Zariski closure of the orbit of ${\mathcal A}$ (under the base-change action of the general linear group). The study of degenerations of algebras is very rich and closely related to deformation theory, in the sense of Gerstenhaber \cite{ger63}. Degenerations have also been used to study a level of complexity of an algebra~\cite{gorb93}. There are many results concerning degenerations of algebras of small dimensions in a variety defined by a set of identities (see, for example, \cite{GRH,    BC99, ikp21,  klp20,ckls, fF68,ale3} and references therein). An interesting question is to study those properties which are preserved under degenerations. Recently, Chouhy~\cite{chouhy} proved that in the case of finite-dimensional associative algebras, the $N$-Koszul property is one such property. Concerning Lie algebras, Grunewald--O'Halloran~\cite{GRH} calculated the degenerations for the variety of five-dimensional nilpotent Lie algebras while in~\cite{BC99}, Burde and Steinhoff constructed the graphs of degenerations for the varieties of three and four-dimensional Lie algebras.
Kaygorodov, Lopes and Popov described all degenerations in the variety of five-dimensional associative commutative algebras \cite{klp20}.
Fern\'andez Ouaridi, Kaygorodov, Khrypchenko and Volkov described the full graphs of degenerations of small dimensional nilpotent algebras \cite{fkkv}. 
Alvarez constructed the graph of degenerations of $8$-dimensional $2$-step nilpotent anticommutative algebras \cite{ale3}.
One of the main problems of the {\it geometric classification} of a variety of algebras is a description of its irreducible components. In~\cite{gabriel}, Gabriel described the irreducible components of the variety of four-dimensional unital associative algebras and the variety of five-dimensional unital associative algebras was classified algebraically and geometrically by Mazzola~\cite{maz79}. Later, Cibils~\cite{cibils} considered rigid associative algebras with $2$-step nilpotent radical.   All irreducible components of $2$-step nilpotent (all, commutative and  anticommutative)  algebras have been described in \cite{shaf,ikp21}.

\newpage
In \cite{bbf11} and \cite{ff16} Lie algebras with small dimensional  square  have been considered. 
$2$-step nilpotent Lie algebras are also under consideration \cite{LO14,BDV}.
The main example of these algebras is the one-dimensional central extension of a $2n$-dimensional abelian Lie algebra,
also named as the $(2n+1)$-dimensional Heisenberg algebra.
Many generalizations of Heisenberg Lie algebras are under a certain consideration \cite{lr,mss}.
Let us call a $2$-step nilpotent algebra with one-dimensional square 
a {\it noncommutative Heisenberg  algebra}.
The notion of a noncommutative Heisenberg  algebra appears in  \cite{nhei};
the notion of a Heisenberg Leibniz  (i.e. indecomposable noncommutative Heisenberg)  algebra  appears in \cite{mg21}.

The main aim of the present paper is to study noncommutative Heisenberg algebras, 
which generalize the notion of a Heisenberg Leibniz algebra and of a Heisenberg Lie algebra.
Note that, all noncommutative Heisenberg algebras are associative and Leibniz,
 all anticommutative algebras from this class are Lie.
We describe the system of degenerations of complex five-dimensional noncommutative Heisenberg algebras. 
As a corollary, we have a description of degenerations in the variety of 
complex   four-dimensional anticommutative $3$-ary algebras.

\section*{Noncommutative Heisenberg algebras}
By an “algebra” in this
paper, we mean   simply a finite-dimensional complex vector
space ${\bf V}$ equipped with a “multiplication” given by an arbitrary element
$\mu \in  {\rm Hom}({\bf V}\otimes {\bf V}, {\bf V} ),$ 
i.e. no assumptions such as associativity or commutativity are made on $\mu$.

\subsection*{The geometric classification of    algebras}
Given an $n$-dimensional complex vector space ${\bf V}$, the set ${\rm Hom}({\bf V} \otimes {\bf V},{\bf V}) \cong {\bf V}^* \otimes {\bf V}^* \otimes {\bf V}$ 
is a vector space of dimension $n^3$
(the case of $n$-ary algebras, see in \cite{kv20}). This space inherits the structure of the affine variety $\mathbb{C}^{n^3}.$ 
Indeed, let us fix a basis $e_1,\dots,e_n$ of ${\bf V}$. Then any $\mu\in {\rm Hom}({\bf V} \otimes {\bf V},{\bf V})$ is determined by $n^3$ structure constants $c_{i,j}^k\in\mathbb{C}$ such that
$\mu(e_i\otimes e_j)=\sum_{k=1}^nc_{i,j}^ke_k$. A subset of ${\rm Hom}({\bf V} \otimes {\bf V},{\bf V})$ is {\it Zariski-closed} if it can be defined by a set of polynomial equations in the variables $c_{i,j}^k$ ($1\le i,j,k\le n$).

The general linear group ${\rm GL}({\bf V})$ acts by conjugation on the variety ${\rm Hom}({\bf V} \otimes {\bf V},{\bf V})$ of all algebra structures on ${\bf V}$:
$$ (g * \mu )(x\otimes y) = g\mu(g^{-1}x\otimes g^{-1}y),$$ 
for $x,y\in {\bf V}$, $\mu\in {\rm Hom}({\bf V} \otimes {\bf V},{\bf V})$ and $g\in {\rm GL}({\bf V})$. Clearly, the ${\rm GL}({\bf V})$-orbits correspond to the isomorphism classes of algebras structure on ${\bf V}$. Let $T$ be a set of polynomial identities which is invariant under isomorphism. Then the subset $\mathbb{L}(T)\subset {\rm Hom}({\bf V} \otimes {\bf V},{\bf V})$ of the algebra structures on ${\bf V}$ which satisfy the identities in $T$ is ${\rm GL}({\bf V})$-invariant and Zariski-closed. It follows that $\mathbb{L}(T)$ decomposes into ${\rm GL}({\bf V})$-orbits. The ${\rm GL}({\bf V})$-orbit of $\mu\in\mathbb{L}(T)$ is denoted by $O(\mu)$ and its Zariski closure by $\overline{O(\mu)}$.

Let ${\bf A}$ and ${\bf B}$ be two $n$-dimensional algebras satisfying the identities from $T$ and $\mu,\lambda \in \mathbb{L}(T)$ represent ${\bf A}$ and ${\bf B}$ respectively.
We say that ${\bf A}$ {\it degenerates} to ${\bf B}$ and write ${\bf A}\to {\bf B}$ if $\lambda\in\overline{O(\mu)}$.
Note that in this case we have $\overline{O(\lambda)}\subset\overline{O(\mu)}$. Hence, the definition of degeneration does not depend on the choice of $\mu$ and $\lambda$. It is easy to see that any algebra degenerates to the algebra with zero multiplication. If ${\bf A}\not\cong {\bf B}$, then the assertion ${\bf A}\to {\bf B}$ 
is called a {\it proper degeneration}. We write ${\bf A}\not\to {\bf B}$ if $\lambda\not\in\overline{O(\mu)}$. 

Let ${\bf A}$ be represented by $\mu\in\mathbb{L}(T)$. Then  ${\bf A}$ is  {\it rigid} in $\mathbb{L}(T)$ if $O(\mu)$ is an open subset of $\mathbb{L}(T)$.
Recall that a subset of a variety is called {\it irreducible} if it cannot be represented as a union of two non-trivial closed subsets. A maximal irreducible closed subset of a variety is called an {\it irreducible component}.
It is well known that any affine variety can be represented as a finite union of its irreducible components in a unique way.
The algebra ${\bf A}$ is rigid in $\mathbb{L}(T)$ if and only if $\overline{O(\mu)}$ is an irreducible component of $\mathbb{L}(T)$. 

%We use the following notation: 

%\begin{enumerate}
%\item $Ann_L(A)=\{ a \in A \mid xa =0 \mbox{ for all } x\in A \}$ is the left  annihilator of $A;$
%\item $Ann_R(A)=\{ a \in A \mid ax =0 \mbox{ for all } x\in A \}$ is the right  annihilator of $A;$
%\item $Ann(A)=Ann_R(A)\cap Ann_L(A)$ is the  annihilator of $A;$
%\item $AZ(A)=\{ a \in A \mid xa +ax=0 \mbox{ for all } x\in A \}$ is the anticommutative center of $A;$%
%\item $dA_L(A)$ is the dimension of the left annihilator of $A;$
%\item $dA_R(A)$ is the dimension of the right annihilator of $A;$
%\item $dA(A)$ is the dimension of the annihilator of $A;$
%\item $dAZ(A)$ is  the dimension of the anticommutative center of $A,$ i.e. of the space $$\{a\in A\mid xa +ax=0, \mbox{ for all } x\in A \};$$
%\item
%\item $msub_0(A)$ is a trivial subalgebra of $A$ of the maximal dimension (we fix one for each algebra $A$);
%\item $A^{(+2)}$ is the space $\{xy+yx\mid x,y\in A\}$.
%\item  $Z(A)$ is the center of $A.$%  $\{ a \in A \mid xa =ax, \mbox{ for all } x\in A \}.$
%\end{enumerate}
%Given spaces $U$ and $W$, we write simply $U>W$ instead of $dim\,U>dim\,W$.

%We use also the notation $U\circ W=UW+WU$.

In the present work we use the methods applied to Lie algebras in \cite{BC99,GRH,GRH2,S90}.
First of all, if ${\bf A}\to {\bf B}$ and ${\bf A}\not\cong {\bf B}$, then $\dim \mathfrak{Der}({\bf A})<\dim \mathfrak{Der}({\bf B})$, where $\mathfrak{Der}({\bf A})$ is the Lie algebra of derivations of ${\bf A}$. We will compute the dimensions of algebras of derivations and will check the assertion ${\bf A}\to {\bf B}$ only for such ${\bf A}$ and ${\bf B}$ that $\dim \mathfrak{Der}({\bf A})<\dim \mathfrak{Der}({\bf B})$. Secondly, 
if ${\bf A}\to {\bf C}$ and ${\bf C}\to {\bf B}$ then ${\bf A}\to{\bf  B}$. If there is no ${\bf C}$ such that ${\bf A}\to {\bf C}$ and ${\bf C}\to {\bf B}$ are proper degenerations, then the assertion ${\bf A}\to {\bf B}$ is called a {\it primary degeneration}. 
If $\dim \mathfrak{Der}({\bf A})<\dim \mathfrak{Der}({\bf B})$ and there are no ${\bf C}$ and ${\bf D}$ such that ${\bf C}\to {\bf A}$, ${\bf B}\to {\bf D}$, ${\bf C}\not\to {\bf D}$ and one of the assertions ${\bf C}\to {\bf A}$ and ${\bf B}\to {\bf D}$ is a proper degeneration,  then the assertion ${\bf A} \not\to {\bf B}$ is called a {\it primary non-degeneration}. It suffices to prove only primary degenerations and non-degenerations to describe degenerations in the variety under consideration. 
%From now on we use this fact without mentioning it.
%
%Degenerations of four dimensional and nilpotent five and six dimensional Lie algebras were described in \cite{BC99,S90,GRH}.
%Since the set $\mathbb{L}(T)$ is closed for any $T$, a Lie algebra cannot degenerate to a non-Lie algebra.
%So when we want to add Malcev or BL algebras to Lie algebras we don't have to check the degenerations from Lie algebras to any of the added algebras.

To prove primary 
degenerations, we will construct families of matrices parametrized by $t$. Namely, let ${\bf A}$ and ${\bf B}$ be two algebras represented by the structures $\mu$ and $\lambda$ from $\mathbb{L}(T)$, respectively. Let $e_1,\dots, e_n$ be a basis of ${\bf V}$ and $c_{i,j}^k$ ($1\le i,j,k\le n$) be the structure constants of $\lambda$ in this basis. If there exist $a_i^j(t)\in\mathbb{C}$ ($1\le i,j\le n$, $t\in\mathbb{C}^*$) such that the elements $E_i^t=\sum_{j=1}^na_i^j(t)e_j$ ($1\le i\le n$) form a basis of ${\bf V}$ for any $t\in\mathbb{C}^*$, and the structure constants $c_{i,j}^k(t)$ of $\mu$ in the basis $E_1^t,\dots, E_n^t$ satisfy $\lim\limits_{t\to 0}c_{i,j}^k(t)=c_{i,j}^k$, then ${\bf A}\to {\bf B}$. In this case  $E_1^t,\dots, E_n^t$ is called a {\it parametric basis} for ${\bf A}\to {\bf B}$.

To prove primary non-degenerations we will use the following lemma (see \cite{GRH}).

\begin{lemma} 
Let $\mathcal{B}$ be a Borel subgroup of ${\rm GL}({\bf V})$ and $\mathcal{R}\subset \mathbb{L}(T)$ be a $\mathcal{B}$-stable closed subset.
If ${\bf A} \to {\bf B}$ and ${\bf A}$ can be represented by $\mu\in\mathcal{R}$ then there is $\lambda\in \mathcal{R}$ that represents ${\bf B}$.
\end{lemma}

%In particular, it follows from Lemma \ref{main} that $A\not\to B$ in the following cases:
%\begin{enumerate}
%\item $Ann_L(A)>Ann_L(B)$;
%\item $Ann_R(A)>Ann_R(B)$;
%\item $Ann(A)>Ann(B)$;
%\item $AZ(A)>AZ(B)$;
%\item $msub_0(A)>msub_0(B)$;
%\item $A^{2}<B^{2}$;
%\item $A^{(+2)}<B^{(+2)}$.
%\item  $Z(A)>Z(B)$.
%\end{enumerate}
Each time when we will need to prove some primary non-degeneration $\mu\not\to\lambda$, we will define $\mathcal{R}$ by a set of polynomial equations in structure constants $c_{i,j}^k$ and write some basis $f_1,\dots,f_n$ after the definition of $\mathcal{R}$ in such a way that the structure constants of $\mu$ in the basis $f_1,\dots,f_n$ satisfy these equations. We will omit everywhere the verification of the fact that $\mathcal{R}$ is stable under the action of the subgroup of lower triangular matrices and of the fact that $\lambda\not\in\mathcal{R}$ for any choice of a basis of ${\bf V}.$ For example, if, for $\xi\in\mathbb{C}$ and an algebra ${\bf A}$, we define 
\begin{center}$Z_{\xi}({\bf A})= \{ x\in {\bf A}:  xy=\xi yx, \forall y \in {\bf A} \},$
\end{center}
then ${\bf A} \to {\bf B}$ will imply $Z_{\xi}({\bf A})\le Z_{\xi}({\bf B})$. Here and further we will write $U$ instead of ${\rm dim}_{\mathbb{C}}U$ if it does not cause any confusion.

\subsection*{Noncommutative Heisenberg algebras, anticommutative $n$-ary algebras and matrices} 
 
\begin{definition}
An algebra $\mathfrak{H}$ is a noncommutative Heisenberg algebra if ${\rm dim} \ \mathfrak{H}^2 \leq 1$ and
$\mathfrak{H}^2\mathfrak{H}+\mathfrak{H}\mathfrak{H}^2=0.$ 
\end{definition}

Let now $\mathfrak{H}$ be an $(n+1)$-dimensional noncommutative Heisenberg algebra. 
There is  a basis $\{e_j\}_{1\leq i \leq n+1}$ of $\mathfrak{H}$ such that 
the multiplication table of $\mathfrak{H}$ is $e_ie_j=\alpha_{ij} e_{n+1}$ for $1 \leq i,j\leq n.$
In this way we may construct an $n\times n$ matrix from a noncommutative Heisenberg algebra and vice versa. Two matrices correspond to isomorphic algebras if and only if they are congruent.
Following \cite{fil85}, we will explain below how construct a (one-t-one) correspondence between $n\times n$ matrices and $n$-dimensional anticommutative  $(n-1)$-ary  algebras such that two algebras are isomorphic if and only if the corresponding matrices are congruent. Thus, the variety of $n\times n$ matrices is isomorphic to the variety of $n$-dimensional anticommutative  $(n-1)$-ary  algebras and the variety of $(n+1)$-dimensional noncommutative Heisenberg algebras has a closed subvariety isomorphic to the variety of $n\times n$ matrices that intersects each orbit under the action of the general linear group. This automatically gives an equivalence between the algebraic classification of $n$-dimensional anticommutative  $(n-1)$-ary  algebras, the algebraic classification of $(n+1)$-dimensional noncommutative Heisenberg algebras and the classification of $n\times n$ matrices up to congruence. Moreover, it is explained in \cite[Section 4]{IKV} that this equivalence respects orbit closures, i.e. the algebraic classification and degenerations of $(n+1)$-dimensional noncommutative Heisenberg algebras that we will give in the present paper give automatically  the algebraic classification and degenerations of $n$-dimensional anticommutative  $(n-1)$-ary  algebras and of $n\times n$ matrices up to congruence.

\subsection*{The algebraic classification  of five-dimensional noncommutative Heisenberg  algebras}
According to the method of classification of $n \times n$ matrices under congruence \cite{hs08},  
there are only $16$ types of $4\times 4$ matrices given below.

\begin{longtable}{cccc}
$\begin{pmatrix}
0&0&0&0\\
0&0&0&0\\
0&0&0&0\\
0&0&0&0\\
\end{pmatrix}$,&

$\begin{pmatrix}
1&0&0&0\\
0&0&0&0\\
0&0&0&0\\
0&0&0&0\\
\end{pmatrix}$,&

$\begin{pmatrix}
0&-1&0&0\\
1&1&0&0\\
0&0&0&0\\
0&0&0&0\\
\end{pmatrix}$,&

$\begin{pmatrix}
0&1&0&0\\
\lambda &0&0&0\\
0&0&0&0\\
0&0&0&0\\
\end{pmatrix}$,\\

$\begin{pmatrix}
1&0&0&0\\
0 &0&-1&0\\
0&1&1&0\\
0&0&0&0\\
\end{pmatrix}$,&

$\begin{pmatrix}
1&0&0&0\\
0 &0&1&0\\
0&\lambda&0&0\\
0&0&0&0\\
\end{pmatrix}$,&

$\begin{pmatrix}
0&1&0&0\\
0&0&1&0\\
0&0&0&0\\
0&0&0&0\\
\end{pmatrix}$,&

$\begin{pmatrix}
0&0&1&0\\
0&-1&-1&0\\
1&1&0&0\\
0&0&0&0\\
\end{pmatrix}$,\\

$\begin{pmatrix}0&0&0&-1\\
0&0&1&1\\
0&-1&-1&0\\
1&1&0&0\\
\end{pmatrix}$,&

 $\begin{pmatrix}0&0&1&0\\
0&0&0&1\\
\lambda&1&0&0\\
0&\lambda&0&0\\
\end{pmatrix} ({\lambda \neq -1,0})$,&

$\begin{pmatrix}1&0&0&0\\
0&0&1&0\\
0&0&0&1\\
0&0&0&0\\
\end{pmatrix}$,&

$\begin{pmatrix}1&0&0&0\\
0&0&0&1\\
0&0&-1&-1\\
0&1&1&0\\
\end{pmatrix}$,\\

$\begin{pmatrix}0&-1&0&0\\
1&1&0&0\\
0&0&0&-1\\
0&0&1&1\\
\end{pmatrix}$,&

$\begin{pmatrix}0&1&0&0\\
\mu&0&0&0\\
0&0&0&1\\
0&0&\lambda&0\\
\end{pmatrix}$,&

 $\begin{pmatrix}0&1&0&0\\
\lambda&0&0&0\\
0&0&0&-1\\
0&0&1&1\\
\end{pmatrix}$,&

$\begin{pmatrix}0&1&0&0\\
0&0&1&0\\
0&0&0&1\\
0&0&0&0\\
\end{pmatrix}$.
    \end{longtable}
Here  $\lambda$, $\mu$ are determined up to replacement by $\lambda^{-1}$, $\mu^{-1}$ respectively
and up to interchanging $\lambda$ and $\mu$. To omit these coincidences let us introduce the sets
$$
\mathbb{C}_{|\cdot|\le 1}=\big\{ x\in\mathbb{C}:  |x|<1 \big\} \ \cup \ \big\{x\in\mathbb{C}: |x|=1 \mbox{ and } {\rm Im}(x)\geq 0 \big\}
$$
and 
$$
\mathbb{C}_{\ge 0}=\big\{ x\in\mathbb{C}:  {\rm Re}(x)>0 \big\} \ \cup \ \big\{x\in\mathbb{C}: {\rm Re}(x)=0 \mbox{ and } {\rm Im}(x)\geq 0 \big\}.
$$
Then putting the restrictions $\lambda,\mu\in\mathbb{C}_{|\cdot|\le 1}$ and $\lambda-\mu\in \mathbb{C}_{\ge 0}$ on the matrices above we get exactly one representative for each congruence class.
Hence, we obtain the following classification of
nonzero five-dimensional noncommutative Heisenberg algebras.

\begin{theorem}
Let $\mathfrak{H}$ be a nonzero five-dimensional   noncommutative Heisenberg algebra.
Then 
  $\mathfrak{H}$ is isomorphic to exactly one algebra from the following list:

\begin{longtable}{l|l| lllllll}
\hline
$\mathfrak{H}$ & $\mathfrak{Der} \ \mathfrak{H}$ &  \multicolumn{7}{l}{{Multiplication table}}  \\
\hline

$\mathfrak{H}_{01}$ & $17$& $e_1e_1 = e_5$ \\ 
\hline
$\mathfrak{H}_{02}$ & $14$ &$e_1e_2 = -e_5$ & $e_2e_1=e_5$ & $e_2e_2=e_5$\\ 
\hline
$\mathfrak{H}_{03}^\lambda$ & $14+2\delta_{\lambda,-1} $& $e_1e_2 = e_5$ & $e_2e_1=\lambda e_5$\\ 
\hline
$\mathfrak{H}_{04}$ &$10$ & $e_1e_1 = e_5$&  $e_2e_3=-e_5$ & $e_3e_2=e_5$& $e_3e_3=e_5$\\ 
\hline
$\mathfrak{H}_{05}^\lambda$ &$10+2\delta_{\lambda,1}+2\delta_{\lambda,- 1}$& $e_1e_1 = e_5$ & $e_2e_3=  e_5$ & $ e_3e_2=\lambda e_5$\\ 
\hline
$\mathfrak{H}_{06}$ &$11$& $e_1e_2 = e_5$ & $e_2e_3=e_5$ \\ 
\hline
$\mathfrak{H}_{07}$ &$10$& $e_1e_3 = e_5$ & $e_2e_2=-e_5$ & $e_2e_3=-e_5$\\
&& $e_3e_1=e_5$ &  $e_3e_2=e_5$\\

\hline
$\mathfrak{H}_{08}$ &$7$&

$e_1e_4 = -e_5$ & $e_2e_3=e_5$ & $e_2e_4=e_5$ & $e_3e_2=-e_5$ \\
&& $e_3e_3=-e_5$  & $e_4e_1=e_5$ & $e_4e_2=e_5$\\

\hline
$\mathfrak{H}_{09}^{\lambda\neq-1,0}$ &$7  + 2\delta_{\lambda,1}$&

$e_1e_3=e_5$ & $e_2e_4=e_5$ & $e_3e_1=\lambda e_5$ \\&& $e_3e_2=e_5$ & $e_4e_2=\lambda e_5$\\

\hline
$\mathfrak{H}_{10}$ &$8$&

$e_1e_1=e_5$ & $e_2e_3=e_5$ & $e_3e_4=e_5$\\

\hline
$\mathfrak{H}_{11}$ &$7$&

$e_1e_1=e_5$ & $e_2e_4=e_5$ & $e_3e_3=-e_5$ \\&& $e_3e_4=-e_5$ & $e_4e_2=e_5$ & $e_4e_3=e_5$\\

\hline
$\mathfrak{H}_{12}$ &$9$&

$e_1e_2=-e_5$ & $e_2e_1=e_5$ & $e_2e_2=e_5$ \\&& $e_3e_4=-e_5$ & $e_4e_3=e_5$ & $e_4e_4=e_5$\\

\hline
$\mathfrak{H}_{13}^{\lambda;\mu}$ & 
$7+\theta_{\lambda,\mu}$ &

$e_1e_2=e_5$ & $e_2e_1=\mu e_5$ & $e_3e_4=e_5$ & $e_4e_3=\lambda e_5$\\

\hline
$\mathfrak{H}_{14}^{\lambda}$ &$7+4\delta_{\lambda,-1}$&

$e_1e_2=e_5$ & $e_2e_1=\lambda e_5$ & $e_3e_4=-e_5$ \\
&&  $e_4e_3=e_5$ & $e_4e_4=e_5$\\

\hline
$\mathfrak{H}_{15}$ &$7$&

$e_1e_2=e_5$ & $ e_2e_3=e_5$ & $e_3e_4=e_5$\\
\hline

\end{longtable}
 
where $\lambda,\mu\in\mathbb{C}_{|\cdot|\le 1}$ and $\lambda-\mu\in \mathbb{C}_{\ge 0}$ and
\begin{longtable}{lcl}
$\theta_{\lambda,\mu}$ & $=$ & $2 \delta_{\lambda,\mu}+2\delta_{\lambda,-1}\delta_{\mu,-1}+2\delta_{\lambda,1}\delta_{\mu,1}+2\delta_{\lambda,-1}+2\delta_{\mu,-1}. $
\end{longtable}

\end{theorem}

\subsection*{Degenerations of five-dimensional noncommutative Heisenberg algebras}
\begin{theorem}\label{theorem}

  The graph of degenerations of the variety of complex five-dimensional noncommutative Heisenberg algebras is the following:

\medskip 

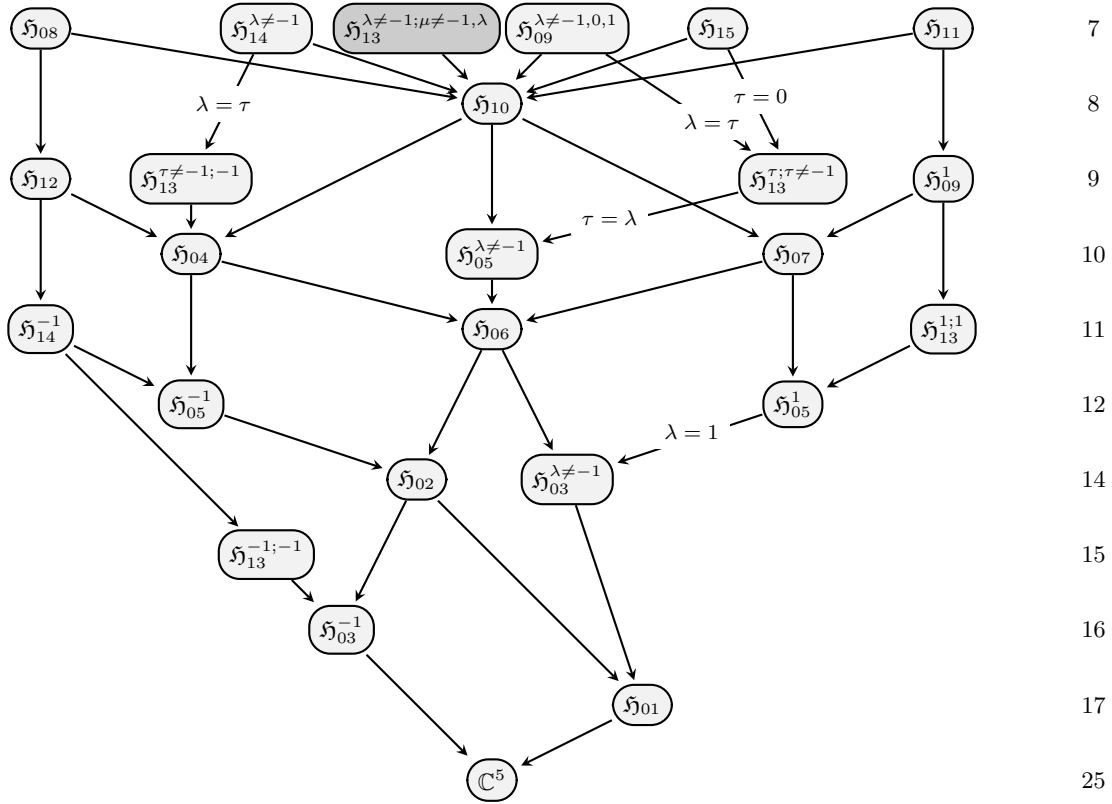
\begin{center}
%	{\bf Figure 1.  The graph of degenerations of five dimensional noncommutative Heisenberg algebras}

	\begin{tikzpicture}[->,>=stealth,shorten >=0.05cm,auto,node distance=1.3cm,
	thick,main node/.style={rectangle,draw,fill=gray!10,rounded corners=1.5ex,font=\sffamily \scriptsize \bfseries },rigid node/.style={rectangle,draw,fill=black!20,rounded corners=1.5ex,font=\sffamily \scriptsize \bfseries },style={draw,font=\sffamily \scriptsize \bfseries }]

	\node (7) at (0,11) {$7$};
	\node (8) at (0,10) {$8$};
	\node (9) at (0,9) {$9$};
	\node (10) at (0,8) {$10$};
	\node (11)  at (0,7) {$11$};
	\node (12)  at (0,6) {$12$};
	\node (14)  at (0,5) {$14$};
	\node (15)  at (0,4) {$15$};
	\node (16)  at (0,3) {$16$};
	\node (17)  at (0,2) {$17$};
	\node (25)  at (0,1) {$25$};

	\node[main node] (h08) at (-14,11) {$\mathfrak{H}_{08}$};
	\node[main node] (h14) at (-11,11) {$\mathfrak{H}_{14}^{\lambda\neq-1}$};
	\node[rigid node] (h13) at (-9,11) {$\mathfrak{H}_{13}^{\lambda\neq-1;\mu\neq-1,\lambda}$};
	\node[main node]  (h09) at (-7,11) {$\mathfrak{H}_{09}^{\lambda\neq -1,0,1 }$};
	\node[main node] (h15) at (-5,11) {$\mathfrak{H}_{15}$};
    \node[main node] (h11) at (-2,11) {$\mathfrak{H}_{11}$};

	\node[main node] (h10) at (-8,10) {$\mathfrak{H}_{10}$};

	\node (h12)[main node] at (-14,9) {$\mathfrak{H}_{12}$};
	\node (h13l-1)[main node] at (-12,9) {$\mathfrak{H}_{13}^{\tau\neq-1; -1}$};
	\node (h13ll)[main node] at (-4,9) {$\mathfrak{H}_{13}^{\tau; \tau\neq-1}$};
	\node (h091)[main node] at (-2,9) {$\mathfrak{H}_{09}^1$};
	
	\node[main node] (h04) at (-12,8) {$\mathfrak{H}_{04}$};
	\node[main node] (h07) at (-4,8) {$\mathfrak{H}_{07}$};
	\node[main node] (h05) at (-8,8) {$\mathfrak{H}_{05}^{\lambda\neq-1}$};
	
	\node[main node] (h14-1) at (-14,7) {$\mathfrak{H}_{14}^{-1}$};
	\node[main node] (h06) at (-8,7) {$\mathfrak{H}_{06}$};
	\node[main node] (h1311) at (-2,7) {$\mathfrak{H}_{13}^{1;1}$};

	\node[main node] (h051) at (-4,6) {$\mathfrak{H}_{05}^1$};
	\node[main node] (h05-1) at (-12,6) {$\mathfrak{H}_{05}^{-1}$};

	\node[main node] (h02) at (-9,5) {$\mathfrak{H}_{02}$};
	\node[main node] (h03) at (-7,5) {$\mathfrak{H}_{03}^{\lambda\neq-1}$};

	\node[main node] (h13-1-1) at (-11,4) {$\mathfrak{H}_{13}^{-1;-1}$};

	\node[main node] (h03-1) at (-10,3) {$\mathfrak{H}_{03}^{-1}$};
	
	\node[main node] (h01) at (-6,2) {$\mathfrak{H}_{01}$};
	
	\node[main node] (CC)  at (-8,1) {$\mathbb{C}^5$};
	
	\path[every node/.style={font=\sffamily\small}]

 (h01) edge   node[left] {} (CC)

 (h02) edge   node[left] {} (h01)
 (h02) edge   node[left] {} (h03-1)

 (h03) edge   node[left] {} (h01)
   
 (h03-1) edge   node[left] {} (CC)

 (h04) edge [bend right=0]  node[left] {} (h05-1)
 (h04) edge node[above]{}   node{} (h06)

(h05) edge   node{} (h06)

(h051) edge   node[above=3, left=-16, fill=white]{\tiny  $\lambda= 1$}   node{} (h03)

 (h05-1) edge   node[left]{ } (h02)
 
 (h06) edge   node[left] {} (h02)
 (h06) edge   node[left] {} (h03)

 (h07) edge   node[left] {} (h06)
 (h07) edge  [bend right=0]  node[left] {} (h051)

 (h08) edge   node[left] {} (h12)
 (h08) edge  node[left] {} (h10)
 
 (h091) edge node[left] {} (h07)
 (h091) edge node[left] {} (h1311)
 
  (h09) edge node[left] {} (h10)
  (h09) edge [bend right=00] node[above=-7, left=-27, fill=white]{\tiny  $\lambda=\tau$} node{} (h13ll)

 (h10) edge   [bend left=0]   node[left] {} (h04)
 (h10) edge    node[left] {} (h05)
 (h10) edge   [bend left=0]  node[left] {} (h07)

(h11) edge   node[left] {} (h10)
(h11) edge   node[left] {} (h091)

 (h12) edge   node[left] {} (h04)
 (h12) edge   node[left] {} (h14-1)

 (h1311) edge   node[left] {} (h051)
 
 (h13l-1) edge node[left] {} (h04)

 (h13) edge  node[left] {} (h10)

 (h13ll) edge  node[above=0, left=-4,  fill=white]{\tiny  $\tau= \lambda$} (h05)

 (h13-1-1) edge   node[left] {} (h03-1)

 (h14) edge   node[left] {} (h10)
 (h14) edge node[above=0, left=-13, fill=white]{\tiny  $\lambda=\tau$} node{} (h13l-1)

 (h14-1) edge   node[left]{ } (h05-1)
 (h14-1) edge   node[left] {} (h13-1-1)

 (h15) edge node[left] {} (h10)
  (h15) edge node[above=2, left=-17, fill=white]{\tiny  $\tau= 0$} (h13ll)
 
    ;
   
	\end{tikzpicture}
	
\end{center}

\end{theorem}

\begin{proof}
All primary degenerations are proved in the following table:
\begin{longtable}{lll lll}
\hline $\mathfrak{H}_{02}$ &$\to$ & $\mathfrak{H}_{01}$&
$E^t_1=  e_2  $&
$E^t_2=   t e_1 $& 
$E^t_3= e_3 $\\
\multicolumn{3}{l}{ }&$E^t_4= e_4 $&
$E^t_5= e_5$\\

\hline $\mathfrak{H}_{02}$ &$\to$ & $\mathfrak{H}_{03}^{-1}$&
$E^t_1=  \frac{1}{t}e_1  $&
$E^t_2=  -t e_2 $& 
$E^t_3= e_3 $\\
&&&$E^t_4= e_4 $&
$E^t_5= e_5$\\

\hline $\mathfrak{H}_{03}^{\lambda \neq -1}$ &$\to$ & $\mathfrak{H}_{01}$&
$E^t_1= e_1 + \frac{1}{(1 + \lambda) }e_2 $&
$E^t_2= te_2 $&
$E^t_3= e_3 $\\
&&&$E^t_4= e_4 $& 
$E^t_5= e_5 $ \\

\hline $\mathfrak{H}_{04}$ &$\to$ & $\mathfrak{H}_{05}^{-1}$&
$E^t_1= e_1 $&
$E^t_2= \frac{1}{t}e_2 $ &$E^t_3= -te_3 $\\&&&
$E^t_4= e_4 $& 
$E^t_5= e_5 $ \\

\hline $\mathfrak{H}_{04}$ &$\to$ & $\mathfrak{H}_{06}$&
$E^t_1= te_1+ite_2 $&
\multicolumn{2}{l}{$E^t_2= \frac{1}{2t}e_1+ \frac{i}{2t}e_3$}\\
&&&
$E^t_3= te_1-ite_2 $&
$E^t_4= e_4 $& 
$E^t_5= e_5 $ \\

\hline $\mathfrak{H}_{05}^{-1}$ &$\to$ & $\mathfrak{H}_{02}$&
$E^t_1=  e_2 $& 
$E^t_2=  e_1 - e_3 $&
$E^t_3= -t e_3 $\\&&& 
$E^t_4= e_4$& 
$E^t_5= e_5$\\

\hline $\mathfrak{H}_{05}^{1}$ &$\to$ & $\mathfrak{H}_{03}^{1}$&
$E^t_1= e_2 $&
$E^t_2=  e_3 $& 
$E^t_3= t e_1$\\
&&&
$E^t_4= e_4 $& 
$E^t_5= e_5 $ \\

\hline $\mathfrak{H}_{05}^{\lambda\neq\pm 1}$ &$\to$ & $\mathfrak{H}_{06}$&
$E^t_1= te_1-\frac{1}{1-\lambda}e_3 $&
\multicolumn{2}{l}{$E^t_2=  \frac{1}{(1-\lambda)t}e_1+e_2-\frac{1}{(1-\lambda)^2(1+\lambda)t^2}e_3 $}\\ 
&&&
$E^t_3= -\lambda t e_1+\frac{1}{1-\lambda}e_3$&
$E^t_4= e_4 $& 
$E^t_5= e_5 $ \\

\hline $\mathfrak{H}_{06}$ &$\to$ & $\mathfrak{H}_{02}$&
$E^t_1= - e_1 + e_3$&
$E^t_2=  e_2 + e_3 $&
$E^t_3= t e_3$\\&&&
$E^t_4= e_4 $& 
$E^t_5= e_5$\\

\hline $\mathfrak{H}_{06}$ &$\to$ & $\mathfrak{H}_{03}^\lambda$&
$E^t_1=  e_1 + \lambda e_3 $&
$E^t_2=  e_2 $& $E^t_3= t e_3 $\\
&&&
$E^t_4= e_4 $&
$E^t_5= e_5$\\

\hline $\mathfrak{H}_{07}$ &$\to$ & $\mathfrak{H}_{05}^1$&
$E^t_1= ie_2 $&
$E^t_2=  \frac{1}{t}e_1 $&
$E^t_3= te_3$\\
&&&
$E^t_4= e_4 $&
$E^t_5= e_5$\\

\hline $\mathfrak{H}_{07}$ &$\to$ & $\mathfrak{H}_{06}$&
$E^t_1= \frac{t}{4}e_3 $&
\multicolumn{2}{l}{$E^t_2= \frac{2}{t}e_1 + \frac{2}{t}e_2 + \frac{1}{t} e_3 $}\\
&&&$E^t_3= -te_2 - \frac{3t}{4}e_3 $&
$E^t_4= e_4$&
$E^t_5= e_5$\\

\hline $\mathfrak{H}_{08}$ &$\to$ & $\mathfrak{H}_{10}$&
$E^t_1= i e_3 $&
$E^t_2= -te_1+te_2 $&
$E^t_3= \frac{1}{2t}e_4 $\\ 
&&&
$E^t_4= te_1+te_2 $&
$E^t_5= e_5  $\\

\hline $\mathfrak{H}_{08}$ &$\to$ & $\mathfrak{H}_{12}$&
$E^t_1= 2 t e_3 $&
$E^t_2= \frac{1}{2t} e_2 + t e_4 $&
$E^t_3= -2it^2 e_4 $\\
&&&
$E^t_4= -\frac{i}{2t^2} e_1 + ie_3 $&
$E^t_5= e_5  $\\  

\hline $\mathfrak{H}_{09}^{\lambda\not=1}$ &$\to$ & $\mathfrak{H}_{10}$&
\multicolumn{2}{l}{$E^t_1= e_1+\frac{1}{1+\lambda}e_3+\frac{\lambda}{(1-\lambda)(1+\lambda)^2}e_4 $}&
$E^t_2= te_3 $\\ 
& &&
\multicolumn{3}{l}{$E^t_3= \frac{1}{t}e_2-\frac{1}{(1-\lambda)(1+\lambda)^2t}e_3+\frac{1}{(1-\lambda)(1+\lambda)^3t}e_4$}\\
&&& $E^t_4= -\lambda te_3+te_4 $&
$E^t_5= e_5  $\\

\hline $\mathfrak{H}_{09}^\lambda$ &$\to$ & $\mathfrak{H}_{13}^{\lambda;\lambda}$&
$E^t_1= \frac{1}{t}e_1 $&
$E^t_2= te_3 $&
$E^t_3= e_2$\\ 
&&&
$E^t_4= e_4 $&
$E^t_5= e_5  $\\

\hline $\mathfrak{H}_{09}^1$ &$\to$ & $\mathfrak{H}_{07}$&
$E^t_1= \frac{1}{2}e_4 $&
\multicolumn{2}{l}{$E^t_2= \frac{1}{2}e_1-e_3+\frac{1}{2}e_4 $}\\
&&&
$E^t_3= 2e_2$ &
$E^t_4= te_3 $&
$E^t_5= e_5  $\\

\hline $\mathfrak{H}_{10}$ &$\to$ & $\mathfrak{H}_{04}$&
$E^t_1=  e_1$&
$E^t_2=-e_2 +e_4 $&
$E^t_3= e_2+e_3$\\
&&&
$E^t_4= te_4$& 
$E^t_5= e_5$\\

\hline $\mathfrak{H}_{10}$ &$\to$ & $\mathfrak{H}_{05}^\lambda$&
$E^t_1= e_1$&
$E^t_2=  e_2 + \lambda  e_4$&
$E^t_3= e_3 $\\&&&
$E^t_4= t e_4 $&
$E^t_5= e_5$\\ 

\hline $\mathfrak{H}_{10}$ &$\to$ & $\mathfrak{H}_{07}^\lambda$&
$E^t_1= e_2+e_4$&
\multicolumn{2}{l}{$E^t_2=  ie_1-e_2 + e_4$}\\
&&& $E^t_3= e_3 $&
$E^t_4= t e_4 $&
$E^t_5= e_5$\\ 

\hline $\mathfrak{H}_{11}$ &$\to$ & $\mathfrak{H}_{09}^1$&
$E^t_1=  -2te_1$&
$E^t_2=te_4$\\
&&&
\multicolumn{2}{l}{$E^t_3=-\frac{1}{2t}e_1+\frac{1}{2t}e_2-\frac{1}{2t}e_3$}\\&&&
$E^t_4= \frac{1}{t}e_2$& 
$E^t_5= e_5$\\ 

\hline $\mathfrak{H}_{11}$ &$\to$ & $\mathfrak{H}_{10}$&
$E^t_1=  e_1$&
\multicolumn{2}{l}{$E^t_2=te_2-te_3$}\\
&&&
$E^t_3=\frac{1}{2t}e_4$&
$E^t_4= te_2+te_3$& 
$E^t_5= e_5$\\

\hline $\mathfrak{H}_{12}$ &$\to$ & $\mathfrak{H}_{04}$&
$E^t_1=  e_2 $&
$E^t_2= e_3 $&
$E^t_3= e_4 $\\
&&&
$E^t_4= t  e_1 $&
$E^t_5= e_5$\\

\hline $\mathfrak{H}_{12}$ &$\to$ & $\mathfrak{H}_{14}^{-1}$&
$E^t_1= \frac{1}{t}e_1 $&
$E^t_2= - te_2 $&
$E^t_3= e_3 $\\
&&&
$E^t_4= e_4 $&
$E^t_5= e_5$\\ 

\hline $\mathfrak{H}_{13}^{\lambda;-1}$ &$\to$ & $\mathfrak{H}_{04}^{\lambda}$&
\multicolumn{2}{l}{$E^t_1=  \frac{1-\lambda}{1+\lambda}e_2+e_3 + \frac{1}{1 + \lambda} e_4$}&
$E^t_2= -ie_2$\\
&&&\multicolumn{2}{l}{$E^t_3= ie_1-\frac{(1-\lambda)i}{1+\lambda}e_2-ie_3+\frac{i}{1+\lambda}e_4 $}\\
&&&$E^t_4= t e_4 $&
$E^t_5= e_5$\\

\hline $\mathfrak{H}_{13}^{\lambda;\lambda}$ &$\to$ & $\mathfrak{H}_{05}^{\lambda}$&
$E^t_1=  e_3 + \frac{1}{1 + \lambda} e_4$&
$E^t_2= e_1$&
$E^t_3= e_2 $\\
&&&
$E^t_4= t e_4 $&
$E^t_5= e_5$\\

\hline $\mathfrak{H}_{13}^{\lambda\neq -1;\mu\neq -1,\lambda}$ &$\to$ & $\mathfrak{H}_{10}$&
\multicolumn{2}{l}{$E^t_1=  e_1+\frac{1}{1+\mu}e_2+\frac{1-\mu}{(\lambda-\mu)(1-\lambda\mu)t}e_4$}&
$E^t_2= te_2+\frac{1}{\lambda-\mu}e_4$ \\
&&&
\multicolumn{3}{l}{$E^t_3= \frac{1}{(\lambda-\mu)t}e_1-\frac{1}{(1+\mu)(1-\lambda\mu)t}e_2+e_3+\frac{1}{(1+\lambda)(\mu-\lambda)(1-\lambda\mu)t^2}e_4 $}\\
&&&
$E^t_4= -\lambda t e_2+\frac{\mu}{\mu-\lambda}e_4 $&
$E^t_5= e_5$\\

\hline $\mathfrak{H}_{14}^{-1}$ &$\to$ & $\mathfrak{H}_{05}^{-1}$&
$E^t_1= e_4 $&
$E^t_2=  e_1$ &$E^t_3= e_2$\\
&&&
$E^t_4= t e_3$&
$E^t_5= e_5$\\

\hline $\mathfrak{H}_{14}^{\lambda\neq -1}$ &$\to$ & $\mathfrak{H}_{10}$&
\multicolumn{2}{l}{$E^t_1= \frac{2}{(1+\lambda)^2t}e_2+\frac{1-\lambda}{1+\lambda}e_3+e_4 $} &$E^t_2=  \frac{1}{1+\lambda}e_2+te_3$\\
&&& 
\multicolumn{2}{l}{$E^t_3= e_1-\frac{1}{(1+\lambda)^3t^2}e_2-\frac{1}{(1+\lambda)t}e_4$}\\
&&&
$E^t_4= e_2$&
$E^t_5= e_5$\\

\hline $\mathfrak{H}_{14}^{\lambda}$ &$\to$ & $\mathfrak{H}_{13}^{\lambda;-1}$&
$E^t_1= te_4 $&
$E^t_2=  \frac{1}{t}e_3$&
$E^t_3= e_1$\\&&&
$E^t_4= e_2$&
$E^t_5= e_5$\\

\hline $\mathfrak{H}_{15}$ &$\to$ & $\mathfrak{H}_{13}^{0;0}$&
$E^t_1= \frac{1}{t}e_1 $&
$E^t_2=  te_2$&
$E^t_3= e_3$\\ &&&
$E^t_4=  e_4$&
$E^t_5= e_5$\\

\hline $\mathfrak{H}_{15}$ &$\to$ & $\mathfrak{H}_{10}$&
$E^t_1= e_1+e_2$&
$E^t_2=  te_1$\\
&&&
$E^t_3= \frac{1}{t}e_2-\frac{1}{t}e_3-\frac{1}{t}e_4$ &$E^t_4=  -te_4$&
$E^t_5= e_5$\\
 \end{longtable}

All primary non-degenerations are proved in the following table:
\begin{longtable}{|rcl|l|}
\hline
\multicolumn{3}{|c|}{$\textrm{Non-degeneration}$} & \multicolumn{1}{|c|}{$\textrm{Arguments}$}\\
\hline
\hline
%$\mathfrak{H}_{07}$ &$\not\to$ & $\mathfrak{H}_{05}^{-1}$&
%$\begin{array}{l}{\mathcal R}=\left\{\begin{array}{l}
%c^5_{1,4}=c^5_{2,3}=c^5_{2,4}=c^5_{3,2}=c^5_{3,3}\\
%=c^5_{3,4}=c^5_{4,1}=c^5_{4,2}=c^5_{4,3}=c^5_{4,4}\\
%=c^5_{3,1}-c^5_{1,3}=0
%\end{array}  \right\}\\
%e_3,e_2,e_1,e_4,e_5\end{array}$ \\
%\hline
$\mathfrak{H}_{08}$ &$\not\to$ & $\mathfrak{H}_{13}^{\lambda\not=-1;-1}, \mathfrak{H}_{13}^{\lambda;\lambda\not=-1}$&
$\begin{array}{l}{\mathcal R}=\left\{\begin{array}{l}
c_{i,j}^k=0\mbox{ if $i=5$, $j=5$ or $k\not=5$;}\\
c^5_{2,4}=c^5_{3,3}=c^5_{3,4}=c^5_{4,2}=c^5_{4,3}=c^5_{4,4}\\
\multicolumn{1}{r}{=c^5_{4,1}+c^5_{1,4}=c^5_{3,2}+c^5_{2,3}=0}
\end{array}  \right\}\\
e_4,e_3,e_2,e_1,e_5\end{array}$ \\
\hline
$\mathfrak{H}_{09}^{\lambda}$ &$\not\to$ & $\mathfrak{H}_{13}^{\mu;-1}, \mathfrak{H}_{13}^{\mu;\mu\not=\lambda}$&
$\begin{array}{l}{\mathcal R}=\left\{\begin{array}{l}
c_{i,j}^k=0\mbox{ if $i=5$, $j=5$ or $k\not=5$;}\\
c^5_{2,4}=c^5_{3,3}=c^5_{3,4}=c^5_{4,2}=c^5_{4,3}=c^5_{4,4}\\
\multicolumn{1}{r}{=c^5_{4,1}-\lambda c^5_{1,4}=c^5_{3,2}-\lambda c^5_{2,3}=0}
\end{array}  \right\}\\
e_2,e_1,e_3,e_4,e_5\end{array}$ \\
\hline
$\mathfrak{H}_{09}^1$ &$\not\to$ & $\mathfrak{H}_{05}^{\lambda\not=1}$
&
$\begin{array}{l}{\mathcal R}=\left\{\begin{array}{l}
c_{i,j}^k=0\mbox{ if $i=5$, $j=5$ or $k\not=5$;}\\
c^5_{2,4}=c^5_{3,3}=c^5_{3,4}=c^5_{4,2}=c^5_{4,3}=c^5_{4,4}\\
\multicolumn{1}{r}{=c^5_{3,1}-c^5_{1,3}=c^5_{4,1}-c^5_{1,4}=c^5_{3,2}-c^5_{2,3}=0}
\end{array}  \right\}\\
e_2,e_3,e_1,e_4,e_5\end{array}$ \\
\hline
$\mathfrak{H}_{11}$ &$\not\to$ & $\mathfrak{H}_{13}^{\lambda;-1}, \mathfrak{H}_{13}^{\lambda;\lambda\not=1}
$&
$\begin{array}{l}{\mathcal R}=\left\{\begin{array}{l}
c_{i,j}^k=0\mbox{ if $i=5$, $j=5$ or $k\not=5$;}\\
c^5_{2,4}=c^5_{3,3}=c^5_{3,4}=c^5_{4,2}=c^5_{4,3}=c^5_{4,4}\\
\multicolumn{1}{r}{=c^5_{4,1}-c^5_{1,4}=c^5_{3,2}-c^5_{2,3}=0}
\end{array}  \right\}\\
e_4,e_1,e_1+ie_3,e_2,e_5\end{array}$ \\
\hline
$\mathfrak{H}_{12},\mathfrak{H}_{13}^{\lambda;-1}$ &$\not\to$ & $\mathfrak{H}_{05}^{\lambda\not=-1}$&
$Z_{-1}\big(\mathfrak{H}_{12}\big),Z_{-1}\big(\mathfrak{H}_{13}^{\lambda;-1}\big)>Z_{-1}\big(\mathfrak{H}_{05}^{\lambda}\big)$ \\
\hline
$\mathfrak{H}_{13}^{-1;-1}$ &$\not\to$ & $\mathfrak{H}_{01}$&
$\mathfrak{H}_{13}^{-1;-1}$ is antisymmetric \\
\hline
$\mathfrak{H}_{13}^{1;1}$ &$\not\to$ & $\mathfrak{H}_{03}^{\lambda\not=1}$&
$\mathfrak{H}_{13}^{1,1}$ is symmetric \\
\hline
$\mathfrak{H}_{13}^{\lambda;\lambda}$ &$\not\to$ & $\begin{array}{l}\mathfrak{H}_{05}^{\mu}\\
\mu\not=\lambda\end{array}$&
$Z_{\lambda}\big(\mathfrak{H}_{13}^{\lambda;\lambda}\big)>Z_{\lambda}\big(\mathfrak{H}_{05}^\mu\big)$ \\
\hline
$ \mathfrak{H}_{13}^{\lambda\neq -1;\mu\neq -1,\lambda}
 $ &$\not\to$ & $\mathfrak{H}_{13}^{\tau;-1},\mathfrak{H}_{13}^{\tau;\tau}$&
$\begin{array}{l}{\mathcal R}=\left\{\begin{array}{l}
c_{i,j}^k=0\mbox{ if $i=5$, $j=5$ or $k\not=5$;}\\
c^5_{2,4}=c^5_{3,3}=c^5_{3,4}=c^5_{4,2}=c^5_{4,3}=c^5_{4,4}\\
\multicolumn{1}{r}{=c^5_{4,1}-\lambda c^5_{1,4}=c^5_{3,2}-\mu c^5_{2,3}=0}
\end{array}  \right\}\\
e_3,e_1,e_2,e_4,e_5\end{array}$ \\
\hline
$\mathfrak{H}_{14}^{-1}$ &$\not\to$ & $\begin{array}{l}\mathfrak{H}_{03}^{\lambda}\\
\lambda\not=-1\end{array}$&
$Z_{-1}\big(\mathfrak{H}_{14}^{-1}\big)>Z_{-1}\big(\mathfrak{H}_{03}^{\lambda}\big)$ \\
\hline
$\mathfrak{H}_{14}^{\lambda\not=-1}$ &$\not\to$ & $\mathfrak{H}_{13}^{\mu\not=\lambda;-1}
,\mathfrak{H}_{13}^{\mu;\mu}$&
$\begin{array}{l}{\mathcal R}=\left\{\begin{array}{l}
c_{i,j}^k=0\mbox{ if $i=5$, $j=5$ or $k\not=5$;}\\
c^5_{2,4}=c^5_{3,3}=c^5_{3,4}=c^5_{4,2}=c^5_{4,3}=c^5_{4,4}\\
\multicolumn{1}{r}{=c^5_{4,1}-\lambda c^5_{1,4}=c^5_{3,2}+ c^5_{2,3}=0}
\end{array}  \right\}\\
e_1,e_4,e_3,e_2,e_5\end{array}$ \\
\hline
$\mathfrak{H}_{15}$ &$\not\to$ & $\mathfrak{H}_{13}^{\lambda;-1},\mathfrak{H}_{13}^{\lambda;\lambda\not=0}$&
$\begin{array}{l}{\mathcal R}=\left\{\begin{array}{l}
c_{i,j}^k=0\mbox{ if $i=5$, $j=5$ or $k\not=5$;}\\
c^5_{2,4}=c^5_{3,3}=c^5_{3,4}=c^5_{4,2}=c^5_{4,3}\\
\multicolumn{1}{r}{=c^5_{4,4}=c^5_{4,1}=c^5_{3,2}=0}
\end{array}  \right\}\\
e_3,e_1,e_2,e_4,e_5\end{array}$ \\
\hline
\end{longtable}

\end{proof}

\subsection*{The algebraic classification  of four dimensional anticommutative $3$-ary  algebras ($\mathfrak{ACom}^3_4$)}
According to the method of classification of $n \times n$ matrices under congruence \cite{hs08} 
and observations from \cite{fil85},
there are only $15$ types of nonzero complex four dimensional anticommutative $3$-ary  algebras.
Let us fix the basis $e_1,e_2,e_3,e_4$ of $\bf V$. 
Any structure $\mu\in\mathfrak{ACom}^3_4$ with structure constants $c_{i_1,i_2,i_3}^j$ ($1\le i_k,j\le 4$) is determined by the $4\times 4$ matrix $A^{\mu}$ whose $(i,j)$-entry is $(-1)^{i-1}c_{i_1,i_2,i_3}^j$, where $(i_1,i_2,i_3)$ is a unique triple of numbers such that $i_1,i_2,i_3\in\{1,2,3,4\}\setminus \{i\}$ and $i_1<i_2<i_3$.

Summarizing, the table below presents an algebraic classification of algebras from $\mathfrak{ACom}^3_4:$

\begin{longtable}{l|l lllllll}
\hline
$\mathfrak{H}$ &  \multicolumn{7}{l}{{Multiplication table}}  \\
\hline

$\mathfrak{H}_{01}$ &  $[e_2,e_3,e_4] = e_1$ \\ 
\hline
$\mathfrak{H}_{02}$ &  $[e_1,e_3,e_4] = -e_1-e_2$ & $[e_2,e_3,e_4] = -e_2$ \\ 
\hline
$\mathfrak{H}_{03}^\lambda$ & $[e_1,e_3,e_4] = -\lambda  e_1$ &  $[e_2,e_3,e_4] = e_2$ \\ 
\hline
$\mathfrak{H}_{04}$& $[e_1,e_2,e_4] = e_2+ e_3$  &  $[e_1,e_3,e_4] = e_3$ &  $[e_2,e_3,e_4] = e_1$\\ 
\hline
$\mathfrak{H}_{05}^\lambda$  & $[e_1,e_2,e_4] = \lambda e_2$ & $[e_1,e_3,e_4] = -e_3$ &  $[e_2,e_3,e_4] = e_1$\\ 
\hline
$\mathfrak{H}_{06}$ & $[e_1,e_3,e_4] = -e_3$&  $[e_2,e_3,e_4] = e_2$  \\ 
\hline
$\mathfrak{H}_{07}$ & $[e_1,e_2,e_4] = -e_1-e_2$ & $[e_1,e_3,e_4] = e_2+e_3$&  $[e_2,e_3,e_4] = e_3$\\

\hline
$\mathfrak{H}_{08}$ & 

$[e_1,e_2,e_3] = -e_1-e_2$  & 
$[e_1,e_2,e_4] = -e_2-e_3$  \\& 
$[e_1,e_3,e_4] = -e_3-e_4$  & 
$[e_2,e_3,e_4] = -e_4$  \\

\hline
$\mathfrak{H}_{09}^{\lambda\neq-1,0}$ & 

$[e_1,e_2,e_3] = -\lambda e_2$ &
$[e_1,e_2,e_4] =\lambda  e_1+ e_2$ \\&
$[e_1,e_3,e_4] = -e_4$  & 
$[e_2,e_3,e_4] = e_3$&\\

\hline
$\mathfrak{H}_{10}$ & 
  $[e_1,e_2,e_4] =  e_4$& 
   $[e_1,e_3,e_4] = -e_3$ &
   $[e_2,e_3,e_4] = e_1$\\

\hline
$\mathfrak{H}_{11}$ & 

$[e_1,e_2,e_3] = -e_2-e_3$ & 
$[e_1,e_2,e_4] = -e_3-e_4$  \\&
$[e_1,e_3,e_4] = -e_4$ & 
$[e_2,e_3,e_4] = e_1$ \\ 

\hline
$\mathfrak{H}_{12}$ & 

$[e_1,e_2,e_3] = -e_3-e_4$&
$[e_1,e_2,e_4] = -e_4$\\& 
$[e_1,e_3,e_4] = -e_1 -e_2$ & 
$[e_2,e_3,e_4] = -e_2$ \\

\hline
$\mathfrak{H}_{13}^{\lambda;\mu}$ & 
 
$[e_1,e_2,e_3] = -\lambda e_3$&
$[e_1,e_2,e_4] = e_4$ \\& 
$[e_1,e_3,e_4] = -\mu e_1$ & 
$[e_2,e_3,e_4] = e_2$ \\

\hline
$\mathfrak{H}_{14}^{\lambda}$ & 
$[e_1,e_2,e_3] = -e_3-e_4$ &
$[e_1,e_2,e_4] = -e_4$  \\ &
$[e_1,e_3,e_4] =-\lambda  e_1$& 
$[e_2,e_3,e_4] = e_2$ \\

\hline
$\mathfrak{H}_{15}$ & 
$[e_1,e_2,e_4] = e_4$&
$[e_1,e_3,e_4] = -e_3$& 
$[e_2,e_3,e_4] = e_2$\\ 
\hline

\end{longtable}
 
where the rest of nonzero multiplications can be obtained by the anticommutative property, $\lambda,\mu\in\mathbb{C}_{|\cdot|\le 1}$ and $\lambda-\mu\in \mathbb{C}_{\ge 0}$.

\begin{corollary}
The graph of degenerations of four-dimensional anticommutative $3$-ary algebras is given in Theorem \ref{theorem}.
\end{corollary}


\begin{thebibliography}{99}


 


\bibitem{ale3}
Alvarez M.A., 
Degenerations of $8$-dimensional $2$-step nilpotent Lie algebras,
Algebras and Represention Theory, 24 (2021),   5, 1231--1243.



\bibitem{bbf11} 
Bartolone C., Di Bartolo A., Falcone G., 
    Nilpotent Lie algebras with $2$-dimensional commutator ideals,  
    Linear Algebra and Its Applications, 434 (2011), 3, 650--656.


\bibitem{nhei} 
Bhattacharya K., Majhi B., 
    Noncommutative Heisenberg algebra in the neighbourhood of a generic null surface,
    Nuclear Physics B, 934 (2018), 557--577.


\bibitem{BC99} 
Burde D., Steinhoff C.,
    Classification of orbit closures of $4$--dimensional complex Lie algebras,
    Journal of Algebra, 214 (1999), 2, 729--739.

\bibitem{BDV} 
Burde D., Dekimpe K., Verbeke B.,  
    Almost inner derivations of $2$-step nilpotent Lie algebras of genus $2$, 
    Linear Algebra and Its Applications, 608 (2021), 185--202.

   \bibitem{ckls}
Camacho L., Kaygorodov I., Lopatkin V., Salim M., 
    The variety of dual mock-Lie algebras,  
    Communications in Mathematics, 28 (2020), 2, 161--178. 

   \bibitem{chouhy}
Chouhy S.,
    On geometric degenerations and Gerstenhaber formal deformations,
    Bulletin of the London Mathematical Society, 51 (2019),  5, 787--797.

   \bibitem{cibils}  Cibils C., 
    $2$-nilpotent and rigid finite-dimensional algebras,
    Journal of the London Mathematical Society (2), 36 (1987), 2, 211--218.


\bibitem{ff16}
    Falcone G., Figula Á., 
    The action of a compact Lie group on nilpotent Lie algebras of type $\{n,2\},$
    Forum Mathematicum, 28 (2016), 4, 795--806.




\bibitem{fkkv}
 Fernández Ouaridi A., Kaygorodov I., Khrypchenko M., Volkov Yu.,  Degenerations of nilpotent algebras, 
 Journal of Pure and Applied Algebra, 226 (2022), 3, 106850.
 
 
 
\bibitem{fil85}
  Filippov V., 
  $n$-Lie algebras, 
  Siberian Mathematical Journal, 26 (1985), 6,  879--891.
  
  
\bibitem{fF68}
Flanigan F.\ J., 
    Algebraic geography: {V}arieties of structure constants, 
    Pacific Journal of Mathematics, 27 (1968), 71--79.
    
     \bibitem{gabriel}
Gabriel P.,
Finite representation type is open,
Proceedings of the International Conference on Representations of Algebras (Carleton Univ., Ottawa, Ont., 1974), pp. 132--155.
 
 \bibitem{ger63}
Gerstenhaber M.,
    On the deformation of rings and algebras,
    Annals of Mathematics (2), 79 (1964), 59--103.
	
   
 

\bibitem{gorb93} 
Gorbatsevich V., 
    Anticommutative finite-dimensional algebras of the first three levels of complexity, 
    St. Petersburg Mathematical Journal, 5 (1994), 3, 505--521.



 


\bibitem{GRH}
Grunewald F.,  O'Halloran J.,
    Varieties of nilpotent Lie algebras of dimension less than six,
    Journal of Algebra, 112 (1988), 2, 315--325.

\bibitem{GRH2}
Grunewald F., O'Halloran J.,
    A Characterization of orbit closure and applications,
    Journal of Algebra, 116 (1988), 1, 163--175.

 

\bibitem{hs08}
Horn R., Sergeichuk V.,
    Canonical matrices of bilinear and sesquilinear forms,
    Linear Algebra and its Applications, 428 (2008), 193--223.



\bibitem{IKV}
Ismailov N., Kaygorodov I.,  Volkov Yu.,
    Degenerations of Leibniz and anticommutative algebras,
 Canadian Mathematical Bulletin,   62   (2019), 3, 539--549.

 
  
  

\bibitem{ikp21} 
Ignatyev M.,   Kaygorodov I.,  Popov Yu.,
The geometric classification of 2-step nilpotent algebras and applications, 
     Revista Matemática Complutense, 35 (2022),   3, 907--922. 
 


\bibitem{klp20}
Kaygorodov I., Lopes S., Popov Yu., 
    Degenerations of nilpotent associative commutative algebras, 
    Communications in Algebra, 48 (2020), 4, 1632--1639.


 
\bibitem{kv20}
Kaygorodov I., Volkov Yu., 
    Degenerations of Filippov algebras, 
    Journal of Mathematical Physics, 61 (2020), 2, 021701.


\bibitem{LO14}
Lauret J.,  Oscari D., 
On non-singular $2$-step nilpotent Lie algebras, 
Mathematical Research Letters, 21 (2014), 3, 553--583.



\bibitem{lr}
Lopes S., Razavinia F.,
Quantum generalized Heisenberg algebras and their representations, 
Communications in Algebra, 50 (2022), 2, 463--483.


\bibitem{mg21}
 Mancini M.,  La Rosa G.,
    Two-step nilpotent Leibniz algebras,
    Linear Algebra and Its Applications, 637 (2022), 119--137.




\bibitem{maz79}
Mazzola G.,
    The algebraic and geometric classification of associative algebras of dimension five, 
    Manuscripta Mathematica, 27 (1979), 1, 81--101. 
    
\bibitem{mss}
Meljanac S., Škoda Z., Štrajn R., 
    Generalized Heisenberg algebra, realizations of the ${\rm gl}(N)$ algebra and applications,  
    Reports on Mathematical Physics, 89 (2022), 1, 131--140.


\bibitem{NR66}
Nijenhuis A., Richardson R.\ W.  Jr.,
    Cohomology and deformations in graded {L}ie algebras, 
    Bulletin of the American Mathematical Society, 72 (1966), 1--29. 
 
 
\bibitem{shaf}
    Shafarevich I., 
    Deformations of commutative algebras of class $2,$ Leningrad Mathematical Journal, 2 (1991), 6, 1335--1351.

\bibitem{S90}
Seeley C., 
    Degenerations of $6$-dimensional nilpotent Lie algebras over $\mathbb{C}$, 
    Communications in Algebra,  18 (1990), 10, 3493--3505.
    

 
 
 


\end{thebibliography}
\end{document}